\documentclass[12pt]{article}
\usepackage[english]{babel}
\usepackage[T1]{fontenc}
\usepackage{dsfont}

\usepackage[T2A]{fontenc}
\usepackage[utf8]{inputenc}

\usepackage{latexsym}
\usepackage{amssymb}
\usepackage{stmaryrd}
\usepackage{amsmath}
\usepackage{paralist} 
\usepackage[dvips]{graphicx,epsfig}
\usepackage{color}

\usepackage{amsmath,amsthm,amssymb}
\usepackage{hyperref}
\usepackage{textcomp}

\newtheorem{theorem}{Theorem}
\newtheorem{lemma}{Lemma}
\newtheorem{proposition}{Proposition}
\newtheorem{corollary}{Corollary}

\newtheorem{statement}{Claim}

\def\leq{\leqslant}
\def\geq{\geqslant}

\date{}
\title{Polynomial method for perfect 2-colourings of circulant graphs}

\author{Svyatoslav Novikov}
\begin{document}
\maketitle

\begin{abstract}
    In this paper we prove that if an infinite circulant graph with $k$
distances has a perfect $2$-colouring with parameters $(b, c)$, then $b + c
\leq 2k + \frac{b+c}{q^t}$ for all positive integers $t$ and primes
$q$ satisfying $\frac{b+c}{gcd(b,c)}\vdots q^t$. In addition, we show that if $b + c = q^t$,
then this necessary condition becomes sufficient for the existence of
perfect $2$-colourings in circulant graphs. 
\end{abstract}

\section{Introduction}
A \emph{perfect $2$-colouring} of a regular graph $G$ with parameters $(b,\,c)$ is a colouring of its vertices in $2$ colours (black and white), in which each black vertice has exactly $b$ white neighbours and each white vertice has $c$ black neighbours.

Perfect colourings are often referred to as equitable partitions; this term was introduced by Delsarte in the book \cite{delsarte}.

An \emph{infinite circulant graph with $k$ distances} $l_1,...,l_k$ is a graph (possibly, with loops and multiple edges), whose vertices are integer numbers; numbers, which differ by $l_i$ for some $i$, are connected with an edge. Denote such a graph by $C_{\infty}(l_1,...,l_k)$. Remark that $C_{\infty}(l_1,...,l_k)$ is a regular graph of degree $2k$.

Perfect $2$-colourings of circulant graphs and their parameters are being subject of active research (see, e.g.,  \cite{khoroshilova09}, \cite{khoroshilova11}, \cite{parshina},\cite{parshina17}, \cite{parshina20}). However, the above works  consider only the cases when the distances $l_1,...,l_k$ have some special form. 

On the contrary, in this work we prove some inequalities between permissible values of $b,c,k$, which apply to arbitrary values of $l_1,...,l_k$. In particular, we prove the hypothesis (stated in \cite{khoroshilova09}) that the parameters $(5,3)$ are not permissible for $3$ distances.

For this purpose we introduce the notion of \emph{multitiling} of an abelian group, which is a natural generalization of the notion of tiling. Next, for the group $\mathbb{Z}/P\mathbb{Z}$ we describe multitilings in terms of polynomials with integer coefficients, which satisfy some divisibility condition, and, in terms of cyclotomic polynomials, we obtain a necessary and sufficient condition for the existence of a multitiling of fixed multiplicity with some fixed "tile". 
One can show that the question of permissibility of the parameters $b,c,k$ for perfect $2$-colourings of graphs $C_{\infty}(l_1,...,l_k)$ can be reduced to the same question for graphs $G_P=C_P(l_1,...,l_k)$ on the residues modulo $P$ for $P$. Finally, perfect $2$-colourings of graphs $C_P(l_1,...,l_k)$ are represented as particular instances of multitilings of multiplicity $c$ of the group $\mathbb{Z}/P\mathbb{Z}$ with some tile $u_{l_1,...,l_k;b,c;P}$, which allows to deduce an inequality on $b,c,k$, which is the main result of the work. Moreover, we show that if $b+c$ is a prime power then the same condition is also a sufficient condition for the permissibility of the parameters $b,c$ for $k$ distances.

Remark that in some other works (see, e.g., \cite{coven}, \cite{fugledepnq}, \cite{pompeiu}, \cite{laba}, \cite{music}, \cite{spectral}) similar reformulations of tilings in terms of polynomials are introduced, and, moreover, similarly to this work, cyclotomic polynomials are applied. However, in such works the condition $(T1)$, first introduced in \cite{coven}, is considered and used only for tilings of multiplicity $1$. In this work we generalize $(T1)$ (point $1$ of lemma \ref{multdivlm}) to multitilings and apply it to perfect $2$-colourings.

\section{Preliminaries and main results}
Polynomials $\Phi_n(x)=\prod\limits_{1 \leq k \leq n;\; \gcd(k,n)=1} (x-e^{2 i \pi \frac{k}{n}})$, $n \geq 1,\; n \in \mathbb{Z}$, are called cyclotomic polynomials. Below some of their properties are given:

\begin{proposition}
1) $\Phi_n(x)$ are irreducible in $\mathbb{Q}[x]$ and have integer coefficients.

2) $\Phi_n(1)=1$, if $n>1$, and $n$ is not a prime power.

3) $\Phi_{p^k}(1)=p$, if $k \geq 1$ and $p$ is a prime.

4) $x^n-1 = \prod\limits_{d \mid n} \Phi_d(x)$.

5) $\Phi_{p^k}(x)=\frac{x^{p^k}-1}{x^{p^{k-1}}-1}=\sum\limits_{l=0}^{p-1}x^{p^{k-1}l}$
\end{proposition}

For an abelian group $H$ we will call a \emph{tile} on $H$ an arbitrary function $u:H \to \mathbb{Z}$.

We call an \emph{$m$-multitiling} of a group $H$ with a tile $u$ a function $v:H \to \mathbb{Z}$ such that 
\begin{equation}\label{convol}
\sum\limits_{h \in H} u(g-h)v(h)=m\;
\end{equation}
for each $g\in H$. 
Assume that $m \in \mathbb{Z},\,m\neq 0$.

We call an \emph{$m$-tiling} of a group $H$ with a tile $u$ an $m$-multitiling $v$ of the group $H$ with some tile $u$ such that $v(h) \in \{0,\,1\}$ for each $h \in H$.

Consider the case $H=\mathbb{Z}/P\mathbb{Z}$. Introduce the polynomials 
$$Q_u(x)=\sum\limits_{a=0}^{P-1}u(a)x^a$$
$$Q_v(x)=\sum\limits_{a=0}^{P-1}v(a)x^a$$

(In the works \cite{fugledepnq}, \cite{pompeiu}, \cite{spectral} polynomials similar to $Q_u(x),\,Q_v(x)$ are referred to as mask polynomials; in \cite{laba} as characteristic polynomials; they also appear in \cite{coven}, \cite{music})

Let $M=\max(l_1,...,l_k)$; $l_1,...,l_k$ are nonnegative integers; let $b>0, c>0$ also be integers.

For $g \in \mathbb{Z}/P\mathbb{Z}$ denote by $\delta_g(h)$ a function on $\mathbb{Z}/P\mathbb{Z}$, which is equal to $1$ at $h=g$, $0$ otherwise.
Also introduce on $\mathbb{Z}/P\mathbb{Z}$ the following function:
$$u_{l_1,...,l_k;b,c;P}(h)=(b+c-2k)\delta_M(h)+\sum\limits_{i=1}^{k}(\delta_{M+l_i}(h)+\delta_{M-l_i}(h)).$$

Denote 
$$A(x)= x^M(b+c-2k+\sum\limits_{i=1}^{k}(x^{l_i}+x^{-l_i})).$$
$$S_P(x)=\prod\limits_{n:\; n \mid P, \;\Phi_n(x) \mid A(x)}\Phi_n(x).$$
$$\tilde{S}_P(x)=\prod\limits_{n:\; n \mid P, \;\Phi_n(x) \mid A(x),\;\text{n is a prime power}}\Phi_n(x).$$ 

Denote by $G_P=C_P(l_1,...,l_k)$ a graph on $P$ vertices, obtained from the graph $C_{\infty}(l_1,...,l_k)$ by factorization of vertices modulo $P$.

We can construct a 1-1 correspondence between $2$-colourings of the graph $G_P$ with parameters $(b,\,c)$ and $c$-\emph{tilings} $v$ of the group $\mathbb{Z}/P\mathbb{Z}$ with the tile $u_{l_1,...,l_k;b,c;P}$: consider a graph $\tilde{G}_P$, which is obtained by adding $b+c-2k$ loops in each vertice. A black-and-white colouring of $G_P$ is perfect with parameters $(b,c)$ if and only if in the same colouring, considered as a colouring of the graph $\tilde{G}_P$, each vertice has exactly $c$ black neighbours (white vertices still have $c$ black neighbours, black vertices had $2k-b$ black neighbours, now they have $(b+c-2k)+(2k-b)=c$ black neighbours). Such black-and-white colourings of the graph $\tilde{G}_P$ (we call them "good") are in a 1-1 correspondence with $c$-\emph{tilings} $v$ of the group $\mathbb{Z}/P\mathbb{Z}$ with the tile $u_{l_1,...,l_k;b,c;P}$: if we color as black exactly the elements $g$ of the group $\mathbb{Z}/P\mathbb{Z}$, for which $v(g)=1$, we obtain a "good"\: colouring of the graph $\tilde{G}_P$; if, in turn, we let $v(g)=1$ for all black vertices $g$ and $v(g)=0$ for white vertices of some "good"\: colouring of the graph $\tilde{G}_P$ with parameters $(b,\,c)$, we obtain a $c$-tiling of the group $\mathbb{Z}/P\mathbb{Z}$ with the tile $u_{l_1,...,l_k;b,c;P}$.

Below the main results of the work are given:

\begin{theorem}\label{pergrcol}
1) If the graph $G_P=C_P(l_1,...,l_k)$ has a perfect $2$-colouring with parameters $(b,\,c)$, then $\tilde{S}_P(1)\vdots \frac{b+c}{gcd(b,c)}$.

2) If $P=q^t$ for some prime number $q$ and integer $t>0$, then the graph $G_P=C_P(l_1,...,l_k)$ has a perfect $2$-colouring with parameters $(b,\,c)$ if and only if $\tilde{S}_P(1)\vdots \frac{b+c}{gcd(b,c)}$.
\end{theorem}

\begin{theorem}\label{pergrcolfirst}
If there exists a circulant graph with $k$ distances, which has a perfect $2$-colouring with parameters $(b,\,c)$, then for each prime $q$ and positive integer $t$ such that $\frac{b+c}{gcd(b,c)} \vdots q^t$, it holds that
$b+c \leq 2k+\frac{b+c}{q^t}$.
\end{theorem}

\begin{corollary}
1) No infinite circulant graph with $2$ distances has a perfect $2$-colouring with parameters $(4,3)$. 

2) No infinite circulant graph with $3$ distances has a perfect $2$-colouring with parameters $(5,3),\,(5,4),\,(6,4)$ or $(6,5)$.

3) No infinite circulant graph with $4$ distances has a perfect $2$-colouring with parameters $(6,5),\,(7,4),\,(8,3),\,(7,5),\,(7,6),\,(8,5),\,(8,6)$ or $(8,7)$.
\end{corollary}
Thus, the hypothesis of inadmissibility of parameters $(5,3)$ for $3$ distances, stated in \cite{khoroshilova09}, is proven.

When $b+c$ is a prime power, one can obtain a necessary and sufficient condition for the existence of a circulant graph with $k$ distances, which has a perfect $2$-colouring with parameters $(b,\,c)$:

\begin{theorem}
\label{lasthm}
Let $b+c=q^s$ for some integer $s>0$ and prime $q$. Then there exists a circulant graph with $k$ distances and its perfect $2$-colouring with parameters $(b,\,c)$, if and only if
$b+c \leq 2k+gcd(b,\,c)$.
\end{theorem}

\section{Polynomial method for multitilings}
\begin{proposition}
The condition (\ref{convol}) is equivalent to 
\begin{equation}\label{polyeq}
Q_u(x)Q_v(x)-m\frac{x^P-1}{x-1}\vdots (x^P-1).
\end{equation}
\end{proposition}
\begin{proof}
Remark that
\begin{align*}
Q_u(x)Q_v(x)=\sum\limits_{c=0}^{P-1}\left(\sum\limits_{0 \leq a,b \leq P-1;\;a+b \equiv c \;(mod\;p)}u(a)v(b)\right)x^{a+b}\\ \equiv 
\sum\limits_{c=0}^{P-1}\left(\sum\limits_{0 \leq a,b \leq P-1;\;a+b \equiv c \;(mod\;p)}u(a)v(b)\right)x^{c}\;(mod\;x^P-1)
\end{align*}
On the other hand,
\begin{align*}
m \frac{x^P-1}{x-1}=\sum\limits_{c=0}^{P-1} m x^c
\end{align*}
Hence, the condition (\ref{polyeq}) is equivalent to 
\begin{align*}
\left(\sum\limits_{0 \leq a,b \leq P-1;\;a+b \equiv c \;(mod\;p)}u(a)v(b)\right)=m
\end{align*}
for all $c$, that is, (\ref{convol}).
\end{proof}

A divisibility condition similar to (\ref{polyeq}) is also used in \cite{coven}, \cite{laba}, \cite{music}, \cite{spectral}.

Introduce analogs of the polynomials $S_A$ from \cite{fugledepnq}, which also appear in \cite{coven}, \cite{pompeiu}, \cite{laba}, \cite{music}, \cite{spectral}: 
$$d_u(x)=\prod\limits_{n \mid P,\;\Phi_n(x) \mid Q_u(x)}\Phi_n(x).$$
$$\tilde{d}_u(x)=\prod\limits_{n \mid P,\;\Phi_n(x) \mid Q_u(x),\, n\text{ is a prime power}}\Phi_n(x)$$
\begin{lemma}\label{multdivlm}
1) Let $m \in \mathbb{Z},\,m \neq 0$. Then an $m$-multitiling $v$ of the group $\mathbb{Z}/P\mathbb{Z}$ with a tile $u:\mathbb{Z}/P\mathbb{Z} \to \mathbb{Z}$ exists if and only if $m\cdot \tilde{d}_u(1) \vdots Q_u(1)$.

2) If $m\cdot \tilde{d}_u(1) \vdots Q_u(1)$, $P=q^t$ for some prime $q$ and positive integer $t$ and, in addition, $0<m\leq Q_u(1)$, then there exists an $m$-tiling $v$ of the group $\mathbb{Z}/P\mathbb{Z}$ with the tile $u:\mathbb{Z}/P\mathbb{Z} \to \mathbb{Z}$.
\end{lemma}
\begin{proof}
As $x^P-1 = \prod\limits_{n \mid P}\Phi_n(x)$, where $\Phi_n(x)$ are irreducible over $\mathbb{Q}[x]$ (in particular, they are pairwise coprime), we obtain $gcd(Q_u(x),\,x^P-1)=\prod\limits_{\Phi_n(x) \mid (x^P-1), \Phi_n(x) \mid Q_u(x)} \Phi_n(x)=\prod\limits_{n \mid P,\Phi_n(x) \mid Q_u(x)}\Phi_n(x)=d_u(x)$.
 If $Q_u(1)=0$, then the condition (\ref{polyeq}) does not hold; but if $Q_u(1) \neq 0$, then $(x-1) \nmid d_u(x)$, hence, (\ref{polyeq}) is equivalent to
\begin{equation}\label{polyeq2}
\frac{Q_u(x)}{d_u(x)}Q_v(x)-m\frac{x^P-1}{(x-1)d_u(x)}\vdots \frac{x^P-1}{d_u(x)}.
\end{equation}
From the definition of $d_u$ the polynomials $\frac{Q_u(x)}{d_u(x)}$ and $\frac{x^P-1}{(x-1)d_u(x)}$ are coprime, hence, due to (\ref{polyeq2}), $Q_v(x)\vdots \frac{x^P-1}{(x-1)d_u(x)}$, that is, $Q_v(x)$ is representable as $\frac{x^P-1}{(x-1)d_u(x)}R_v(x)$, where $R_v(x)$ is a polynomial with integer coefficients such that $deg(R_v)+deg\left(\frac{x^P-1}{(x-1)d_u(x)}\right)\leq P-1$. Then (\ref{polyeq2}) is equivalent to $\frac{Q_u(1)}{d_u(1)}R_v(1)=m$. In particular, 
\begin{equation}\label{maindiv}
m\cdot d_u(1) \vdots Q_u(1).
\end{equation}

If $n>1$ and $n$ is not a prime power then $\Phi_n(1)=1$, hence, $d_u(1)=\tilde{d}_u(1)$.

Consequently, point 1) is proved in one direction.

Conversely, remark that if (\ref{maindiv}) is satisfied, then one can take $R_v(x)=\frac{m d_u(1)}{Q_u(1)}$, $Q_v(x)=\frac{m d_u(1)}{Q_u(1)}\frac{x^P-1}{(x-1)d_u(x)}$, which provides an $m$-multitiling of the group $\mathbb{Z}/P\mathbb{Z}$ with the tile $u$.

\vspace{5pt} 

In order to prove the point 2) it is enough to contrust a polynomial $Q_v(x)$, which satisfies (\ref{polyeq}), whose coefficients are equal to either $0$ or $1$. As $Q_v(x)$ can be represented as $ \frac{x^P-1}{(x-1)d_u(x)}R_v(x)$, it is enough to construct $R_v(x)$
with integer coefficients of degree not larger than $P-1-deg\left(\frac{x^P-1}{(x-1)d_u(x)}\right)=deg(d_u(x))$ such that $R_v(1)=\frac{md_u(1)}{Q_u(1)}$, and each coefficient of $\frac{x^P-1}{(x-1)d_u(x)}R_v(x)$ equals either $0$ or $1$. As $d_u(x) \mid \frac{x^{q^t}-1}{x-1}=\prod\limits_{l=1}^{t}\Phi_{q^l}(x)$, there exists $X \subset \{1,...,t\}$ such that 
$$\tilde{d}_u(x)=d_u(x)=\prod\limits_{r \in X} \Phi_{q^r}(x)=
\prod\limits_{r \in X}\sum\limits_{i=0}^{q-1} x^{q^{r-1}\cdot i}.$$

Hence, all coefficients of $d_u(x)$ are equal to either $0$ or $1$. As due to the conditions of the lemma $0<\frac{md_u(1)}{Q_u(1)}\leq d_u(1)$, one can take as $R_v(x)$ a sum of arbitrary $\frac{md_u(1)}{Q_u(1)}$ monomials whose coefficients are equal to $1$ in $d_u(x)$. 
Then $deg(R_v(x)) \leq deg(d_u(x))$.
Moreover, 
$$\frac{x^P-1}{(x-1)d_u(x)}=\left(\prod\limits_{r \in \{1,...,t\}}\sum\limits_{i=0}^{q-1} x^{q^{r-1}\cdot i}\right)/\left(\prod\limits_{r \in X}\sum\limits_{i=0}^{q-1} x^{q^{r-1}\cdot i}\right)=\prod\limits_{r \in \{1,...,t\}\backslash X}\sum\limits_{i=0}^{q-1} x^{q^{r-1}\cdot i},$$
from which the coefficients of $\frac{x^P-1}{(x-1)d_u(x)}$ are nonnegative, hence, for each integer $a\geq 0$, $a<P$, the coefficient of the polynomial $Q_v(x)=R_v(x)\frac{x^P-1}{(x-1)d_u(x)}$ at $x^a$ is a nonnegative integer which does not exceed the coefficient at $x^a$ of the polynomial $d_u(x)\frac{x^P-1}{(x-1)d_u(x)}=\frac{x^P-1}{x-1}$, which, in turn, equals $1$. Consequently, each coefficient of $Q_v(x)$ equals either $0$ or $1$, then for the tile $v$ it holds that $range(v) \subset \{0,\,1\}$. 
\end{proof}
Remark that the condition $m \cdot \tilde{d}_u(1)\vdots Q_u(1)$ is a generalization of $(T1)$ from \cite{coven} to multitilings.

\section{Proofs of main results}
It is a known fact (\cite{khoroshilova09}), that if a perfect $2$-colouring of the graph $C_{\infty}(l_1,...,l_k)$ exists, then it has some period $P$. 
In other words, for this $P$ there exists a perfect $2$-colouring $S$ with parameters $(b,c)$  of the graph $G_P=C_{P}(l_1,...,l_k)$.

Hence, due to the correspondence between perfect colourings of the graph $G_P$ and tilings of the group $\mathbb{Z}/P\mathbb{Z}$, described in Section 2, theorem \ref{pergrcolfirst} is a corollary of the following lemma:

\begin{lemma}\label{tilingdiv}
The following conditions are equivalent:

1) There exist nonnegative integers $l_1,...,l_k$, an integer $P>1$ and a $c$-multitiling of the group $\mathbb{Z}/P\mathbb{Z}$ with the tile $u_{l_1,...,l_k;b,c;P}$.
 
2) For each prime $q$ and positive integer $t$ such that $\frac{b+c}{gcd(b,c)} \vdots q^t$, it holds that 
$b+c \leq 2k+\frac{b+c}{q^t}$.

Moreover, if 2) is satisfied, then in 1) one can take $P=\frac{b+c}{gcd(b,\,c)}$, if $\frac{b+c}{gcd(b,\,c)}$ is odd and $P=2\frac{b+c}{gcd(b,\,c)}$, if $\frac{b+c}{gcd(b,\,c)}$ is even.
\end{lemma}

Since for each nonnegative integer $g$ it holds that $Q_{\delta_g}(x)\equiv x^g\;(mod\;x^P-1)$, where $\delta_g:\mathbb{Z}/P\mathbb{Z} \to \mathbb{Z}$, $\delta_g(h)=1$, if $g\;mod \;P = h$, $\delta_g(h)=0$ otherwise, then for $u=u_{l_1,...,l_k;b,c;P}$ we have $A(x)-Q_u(x)\vdots(x^P-1)$, hence, $d_u(x)=S_P(x),\,\tilde{d}_u(x)=\tilde{S}_P(x)$.

Theorem \ref{pergrcol}, in turn, due to the correspondence between $c$-tilings and perfect colourings with parameters $(b,c)$, described in Section 2, is a corollary of the following lemma:

\begin{lemma}\label{maindivrask}
1) There exists a $c$-multitiling of the group $\mathbb{Z}/P\mathbb{Z}$ with the tile $u_{l_1,...,l_k;b,c;P}$ if and only if $\tilde{S}_P(1) \vdots \frac{b+c}{gcd(b,c)}$.

2) If $P=q^t$ for some prime $q$ and positive integer $t$, then there exists a $c$-tiling of the group $\mathbb{Z}/P\mathbb{Z}$ with the tile $u_{l_1,...,l_k;b,c;P}$.
\end{lemma}
\begin{proof}
Substitute $m=c$, $u=u_{l_1,...,l_k;b,c;P}$ in lemma \ref{multdivlm}, the condition $m \cdot \tilde{d}_u(1)\vdots Q_u(1)$ can be rewritten as $c \cdot \tilde{S}_P(1) \vdots (b+c)$, since $Q_u(1)=b+c$. This, in turn, is equivalent to $\tilde{S}_P(1) \vdots \frac{b+c}{gcd(b,c)}$. The condition $0<m\leq Q_u(1)$ from point 2) of lemma \ref{multdivlm} is also satisfied. 
\end{proof}

\begin{proof}[Proof of lemma \ref{tilingdiv}]
"$\Rightarrow$"\; 
Since $\Phi_{p^k}(1)=p$ for each prime $p$ and integer $k>0$, then 
\begin{align}
\label{s_tilde_1}
\tilde{S}_P(1)=\prod\limits_{(p,k):\;p\text{ простое},\;k>0,\;p^k \mid P,\;\Phi_{p^k}(x) \mid A(x)}p.
\end{align}
From the conditions of lemma \ref{tilingdiv} combined with lemma \ref{maindivrask} it follows that $\tilde{S}_P(1) \vdots q^t$.
Hence, for at least $t$ pairs $(p,k)$ from the product (\ref{s_tilde_1}) it holds that $p=q$, which implies that there exist $0<s_1<...<s_t$ such that for each $1\leq i \leq t$ it holds that $$A(x) \vdots \Phi_{q^{s_i}}(x)=
\frac{x^{q^{s_i}}-1}{x^{q^{s_i-1}}-1}\;(**).$$

Denote $h_{j,r}=\sum\limits_{r':\; q^{j} \mid (r'-r)}a_{r'}$, where $a_{r'}$ is the coefficient of the polynomial $A(x)$ at $x^{r'}$. Then it is easy to see that 
$(**)$ can be rewritten as\\ $h_{s_i,r}=h_{s_i,r+q^{s_i-1}}$ for each $i,r$, since $$(x^{q^{s_i-1}}-1)A(x) \equiv
\sum\limits_{r=0}^{q^{s_i}-1} h_{s_i,r-q^{s_i-1}}\cdot x^r-\sum\limits_{r=0}^{q^{s_i}-1} h_{s_i,r}\cdot x^r\;(mod\;x^{q^{s_i}}-1).$$

For convenience we will consider that $s_0=0$.
\begin{statement}\label{ineqq}
For $1 \leq i \leq t$ it holds that $h_{s_{i-1},M}\geq q\cdot h_{s_i,M}$.
\end{statement} 
\begin{proof}
The claim follows from the next relations: 
\begin{align*}
q \cdot h_{s_i,M} = \sum\limits_{b=0}^{q-1} h_{s_i,M+b \cdot q^{s_i-1}} = h_{s_i-1,M} \leq h_{s_{i-1},M}
\end{align*}
Here the first equality follows from $(**)$.
Let us prove the second equality:
$$h_{s_i-1,M}=\sum\limits_{r': q^{s_i-1} \mid (r'-M)} a_{r'}=\sum\limits_{b=0}^{q-1}\sum\limits_{r': q^{s_i} \mid (r'-M-b q^{s_i-1})} a_{r'}=\sum\limits_{b=0}^{q-1}h_{s_i,M+b\cdot q^{s_i-1}}.$$
 The last inequality follows from the fact that the coefficients of $A(x)$, except for possibly the coefficient at $x^M$, are nonnegative.
\end{proof}
Applying claim \ref{ineqq} $t$ times and again using nonnegativity of the coefficients of $A(x)$, except for possibly the coefficient at $x^M$, we obtain $b+c=h(1)=h_{s_0,M} \geq q^t \cdot h_{s_t,M} \geq q^t\cdot(b+c-2k)$ as required.

"$\Leftarrow$"\; By lemma \ref{maindivrask} it is enough to construct $l_1,...,l_k;P$ such that $\tilde{S}_P(1) \vdots \frac{b+c}{gcd(b,c)}$. In order to do this we will prove the following proposition: 

\begin{proposition} 
\label{l_prim}
Let $\frac{b+c}{gcd(b,c)}=q_1^{t_1}...q_s^{t_s}$ - be the decomposition of $\frac{b+c}{gcd(b,c)}$ into prime multiples.

There exist $l'_{i,\,1},...,l'_{i,\,k}$ such that for each nonnegative integers $l_1,...,l_k$ and $P>1$ which satisfy the following conditions:\\
1) $l_j\equiv l'_{i,\,j}\;(mod\,q_i^{t_i})$ when $q_i>2$,\\
2) $l_j\equiv l'_{i,\,j}\;(mod\,2^{t_i+1})$ when $q_i=2$,\\
3) $M=\max(l_1,...,l_k)>q_i^{t_i+1}$,\\
4) $P\vdots q_i^{t_i}$ when $q_i>2$,\\
5) $P \vdots 2^{t_i+1}$ when $q_i=2$,\\
it holds that $\tilde{S}_P(1) \vdots q_i^{t_i}$.
\end{proposition}
First let us make sure that the "$\Leftarrow$"\; part of lemma \ref{tilingdiv} follows from proposition \ref{l_prim}.
It is enough to apply the Chinese remainder theorem: if $\frac{b+c}{gcd(b,\,c)}$ is odd, one can take $P=\prod\limits_{i=1}^{k} q_i^{t_i}=\frac{b+c}{gcd(b,\,c)}$; if $\frac{b+c}{gcd(b,\,c)}$ is even, one can take $P=2\prod\limits_{i=1}^{k} q_i^{t_i}=2\frac{b+c}{gcd(b,\,c)}$. Next, one can take arbitrary $l_1,...,l_k$ such that $l_j\equiv l'_{i,\,j}\;(mod\,q_i^{t_i})$ when $q_i>2$ and $l_j\equiv l'_{i,\,j}\;(mod\,2^{t_i+1})$ when $q_i=2$, then $\tilde{S}_P(1) \vdots \frac{b+c}{gcd(b,\,c)}$. Increasing some of $l_i$ by $P$ a sufficient number of times, one can satisfy the condition 3).

\begin{proof}[Proof of proposition \ref{l_prim}]
If $q_i>2$, then $b+c-2k \equiv \frac{b+c}{q_i^{t_i}}\;(mod\;2)$ and by the conditions of the lemma \ref{tilingdiv} $b+c-2k \leq \frac{b+c}{q_i^{t_i}}$, therefore, set the values of $l'_{i,\,j}$ (in arbitrary order) so that there are $\frac{b+c}{2q_i^{t_i}}-\frac{b+c-2k}{2}$ zeros among them, and for each integer $r \geq 1$, $r \leq \frac{q_i^{t_i}-1}{2}$ among $l'_{i,\,j}$ there are $\frac{b+c}{q_i^{t_i}}$ values, equal to $r$ among them. In total there are exactly $$\frac{b+c}{2q_i^{t_i}}-\frac{b+c-2k}{2}+\frac{q_i^{t_i}-1}{2}\cdot \frac{b+c}{ q_i^{t_i}}=k$$ values.
Then it will hold that (here $M'=M-\frac{q_i^{t_i}-1}{2}$)
\begin{align*}
A(x)\equiv \frac{b+c}{q_i^{t_i}}(x^M\sum\limits_{r=1}^{(q_i^{t_i}-1)/2} (x^r+x^{-r})+x^M)=\frac{b+c}{q_i^{t_i}}x^{M'}(\sum\limits_{r=0}^{q_i^{t_i}-1}x^r)\\
=
 \frac{b+c}{q_i^{t_i}}x^{M'} \prod\limits_{j=1}^{t_i}\Phi_{q_i^{j}}(x)\;(mod\;x^{q_i^{t_i}}-1),
\end{align*}
since $\prod\limits_{j=1}^{t_i}\Phi_{q_i^{j}}(x)=\prod\limits_{j=0}^{t_i}\Phi_{q_i^{j}}(x)/(x-1)=\frac{x^{q_i^{t_i}}-1}{x-1}=\sum\limits_{r=0}^{q_i^{t_i}-1}x^r$.
Hence, $A(x)\vdots \prod\limits_{j=1}^{t_i}\Phi_{q_i^{j}}(x)$ and consequently also $\tilde{S}_P(x)\vdots \prod\limits_{j=1}^{t_i}\Phi_{q_i^{j}}(x)$.
Taking into account the fact that $\Phi_{q_i^j}(1)=q_i$ when $j>0$, we obtain that when $P\vdots q_i^{t_i}$ it hols that $\tilde{S}_P(1)\vdots q_i^{t_i}$ as required.

Now consider the case when $q_i=2$, but $\frac{b+c}{2^{t_i}}$ is even. Then $b+c-2k \equiv \frac{b+c}{2^{t_i}}\;(mod\;2)$, and by the conditions of the lemma \ref{tilingdiv} $b+c-2k \leq \frac{b+c}{2^{t_i}}$, therefore, set $l'_{i,\,j}$ so that:\\
I) There are $\frac{b+c}{2\cdot 2^{t_i}}-\frac{b+c-2k}{2}$ zeros among them.\\
II) For each integer $r \geq 1$, $r \leq 2^{t_i-1}-1$ there are $\frac{b+c}{2^{t_i}}$ values equal to $r$ among $l'_{i,\,j}$.\\
III) The value $2^{t_i-1}$ appears $\frac{b+c}{2^{t_i+1}}$ times.\\
In total we get exactly $$\frac{b+c}{2\cdot 2^{t_i}}-\frac{b+c-2k}{2}+(2^{t_i-1}-1)\cdot \frac{b+c}{ 2^{t_i}}+\frac{b+c}{2^{t_i+1}}=k$$ values.
Next we can proceed absolutely analogously to the above case: it will hold that (here $M'=M-(2^{t_i-1}-1)$)
\begin{align*}
A(x)\equiv \frac{b+c}{2^{t_i}}(x^M\sum\limits_{r=1}^{2^{t_i-1}-1} (x^r+x^{-r})+x^{M+2^{t_i-1}}+x^M)=\\ \frac{b+c}{2^{t_i}}x^{M'} \sum\limits_{r=0}^{2^{t_i}-1} x^r =
 \frac{b+c}{2^{t_i}}x^{M'} \prod\limits_{j=1}^{t_i}\Phi_{2^{j}}(x)\;(mod\;x^{2^{t_i}}-1),
\end{align*}
as $\prod\limits_{j=1}^{t_i}\Phi_{2^{j}}(x)=\prod\limits_{j=0}^{t_i}\Phi_{2^{j}}(x)/(x-1)=\frac{x^{2^{t_i}}-1}{x-1}=\sum\limits_{r=0}^{2^{t_i}-1}x^r$.
Hence, $A(x)\vdots \prod\limits_{j=1}^{t_i}\Phi_{2^{j}}(x)$ and consequently also $\tilde{S}_P(x)\vdots \prod\limits_{j=1}^{t_i}\Phi_{2^{j}}(x)$.
Taking into consideration the fact that $\Phi_{2^j}(1)=2$ when $j>0$, we obtain that for $P\vdots 2^{t_i}$ it holds that $\tilde{S}_P(1)\vdots 2^{t_i}$ as required.

Finally consider the case when $q_i=2$ and $\frac{b+c}{2^{t_i}}$ is odd: from the conditions of the lemma combined with the fact that $b+c-2k$ is even, we obtain $b+c-2k\leq \frac{b+c}{2^{t_i}}-1$. Introduce the polynomial 
\begin{align*}
R(x)=(\frac{b+c}{ 2^{t_i}}-1)\sum\limits_{j=0}^{2^{t_i}-1}x^{2j}+\sum\limits_{j=0}^{2^{t_i}-1}x^{2j+1}
=\\((\frac{b+c}{ 2^{t_i}}-1)+x)\sum\limits_{j=0}^{2^{t_i}-1}x^{2j}=((\frac{b+c}{ 2^{t_i}}-1)+x)\prod\limits_{j=1}^{t_i}\Phi_{2^{j+1}}(x),
\end{align*}
as $\prod\limits_{j=1}^{t_i}\Phi_{2^{j+1}}(x)=\prod\limits_{j=0}^{t_i+1}\Phi_{2^{j}}(x)/\prod\limits_{j=0}^{1}\Phi_{2^{j}}(x)=\frac{x^{2^{t_i+1}}-1}{x^2-1}=\sum\limits_{j=0}^{2^{t_i}-1}x^{2j}$.
Since $\Phi_{2^j}(1)=2$ when $j>0$, in order for $\tilde{S}_P(1)\vdots 2^{t_i}$ to be satisfied it is enough to take $P \vdots 2^{t_i+1}$ and $l'_{i,\,1},...,l'_{i,\,k}$ such that 
$$A(x) \equiv x^M R(x)\;(mod\;x^{2^{t_i+1}}-1).$$
One can achieve this by taking as $l'_{i,j}$:\\
I) $(\frac{b+c}{2^{t_i}}-1-(b+c-2k))/2$ values equal to $0$ (it is possible since $\frac{b+c}{2^{t_i}}-1-(b+c-2k)$ is even and nonnegative). \\ 
II) All the values of the form $2j+1$, where $0\leq j \leq 2^{t_i-1}-1$, one time each.\\ 
III) All the values of the form $2j$, where $1 \leq j \leq 2^{t_i-1}-1$,  $\frac{b+c}{2^{t_i}}-1$ times each.\\
IV) The value $2^{t_i}$ $(\frac{b+c}{2^{t_i}}-1)/2$ times.\\
Indeed, in total there are
$$(\frac{b+c}{2^{t_i}}-1-(b+c-2k))/2+2^{t_i-1}+(\frac{b+c}{2^{t_i}}-1)(2^{t_i-1}-1)+(\frac{b+c}{2^{t_i}}-1)/2=k$$
values. Next, we obtain
\begin{align*}
x^M R(x) \equiv 
x^M((\frac{b+c}{ 2^{t_i}}-1)+x)(1+x^2+x^4+...+x^{2^{t_i+1}-2}) \equiv\\
x^M\bigg((\frac{b+c}{2^{t_i}}-1)+\sum\limits_{j=0}^{2^{t_i-1}-1}(x^{2j+1}+x^{-2j-1})+(\frac{b+c}{2^{t_i}}-1)\sum\limits_{j=1}^{2^{t_i-1}-1}(x^{2j}+x^{-2j})+\\
(x^{2^{t_i}}+x^{-2^{t_i}})(\frac{b+c}{2^{t_i}}-1)/2\bigg) \equiv A(x)\;(mod\; x^{2^{t_i+1}-1})
\end{align*}
\end{proof}
\end{proof}

Thus, lemma \ref{tilingdiv}, and consequently (as shown in Section 2) theorem \ref{pergrcolfirst} is proven.


\begin{proof}[Proof of theorem \ref{lasthm}]
"$\Rightarrow$": follows from theorem \ref{pergrcolfirst}.

"$\Leftarrow$":
apply lemma \ref{tilingdiv}: construct the corresponding $P;l_1,...,l_k$. One can assume that $P=q^{s'}$ for some $s'>0$. Then by lemma \ref{maindivrask} it holds that $\tilde{S}_P(1)\vdots \frac{b+c}{gcd(b,\,c)}$, and then by point 2) of theorem \ref{pergrcol} the circulant graph with distances $l_1,...,l_k$ has a $P$-periodic perfect $2$-colouring with parameters $(b,\,c)$.
\end{proof}

\end{document}